\documentclass[12pt, a4paper]{amsart}
\usepackage{amsfonts}
\usepackage{amssymb,amsmath,mathtools,xcolor,graphicx,xspace,colortbl,ragged2e,rotating}
\usepackage{pstricks,pstcol,pst-plot,pst-node,pst-tree,amssymb}
\usepackage{pstricks-add}

\graphicspath{{Rossi_spheres_basics_02-18-2019_graphics/}{Rossi_spheres_basics_02-18-2019_tcache/}{Rossi_spheres_basics_02-18-2019_gcache/}}
\DeclareGraphicsExtensions{.pdf,.eps,.ps,.png,.jpg,.jpeg}

\newtheorem{theorem}{Theorem}

\newtheorem{corollary}[theorem]{Corollary}

\newtheorem{lemma}[theorem]{Lemma}

\newtheorem{proposition}[theorem]{Proposition}
\newtheorem{remark}[theorem]{Remark}

\newenvironment{pfn}{\noindent{\em Proof}}{\hfill $\square$ \medskip}
\input{tcilatex}

\begin{document}
\title[Homogeneous CR 3-manifolds]{A characterization of homogeneous
three-dimensional CR manifolds}
\author{Jih-Hsin Cheng}
\address{Institute of Mathematics, Academia Sinica and National Center for
Theoretical Sciences, Taipei, Taiwan, R.O.C.}
\email{cheng@math.sinica.edu.tw}
\author{Andrea Malchiodi}
\address{Scuola Normale Superiore, Piazza dei Cavalieri 7, 50126 Pisa, ITALY}
\email{andrea.malchiodi@sns.it}
\author{Paul Yang}
\address{Princeton University, Department of Mathematics\\
Fine Hall, Washington Road, Princeton NJ 08544-1000 USA}
\email{yang@math.princeton.edu}
\thanks{}

\begin{abstract}
We characterize homogeneous three-dimensional CR manifolds, in particular
Rossi spheres, as critical points of a certain energy functional that
depends on the Webster curvature and torsion of the pseudohermitian
structure.
\end{abstract}

\maketitle

\section{\textbf{Introduction and statement of the results}}

Let $(M,\xi )$ be a contact 3-manifold with contact structure $\xi .$ A CR
structure on $(M,\xi )$ is an endomorphism $J:\xi \rightarrow \xi $
satisfying $J^{2}=Id.$ With a choice of contact form $\theta $ (i.e. $\theta
|_{\xi }$ $=0$, $\theta \wedge d\theta $ $\neq $ $0)$ such that $d\theta
(\cdot ,J\cdot )$ $>$ $0,$ $(J,\theta )$ is called a (strictly pseudoconvex)
pseudohermitian structure. To such a structure one associates the
(Tanaka-Webster) scalar curvature $R$ and the torsion tensor $A_{11}$ with
norm $|A|_{J,\theta }^{2}$ $:=$ $h^{1\bar{1}}h^{1\bar{1}}A_{11}A_{\bar{1}%
\bar{1}}$ (see for instance \cite{Lee} for basic pseudohermitian geometry).
On $(M,\xi )$ $=$ $(S^{3},\hat{\xi}),$ the standard contact 3-sphere, there
exists a family of distinguished homogeneous pseudohermitian structures $%
(J_{(s)},$ $\hat{\theta}),$ called Rossi spheres, where $\hat{\theta}$ is
the standard contact form on $(S^{3},\hat{\xi}).$ See Subsection \ref%
{Subsection1-1} for a detailed description.

We recall that Rossi spheres ($s$ $\neq $ $0)$ are the simplest examples of
non-embeddable CR 3-manifolds which (two to one) cover embeddable ones in $%
\mathbb{C}^{3}$ \cite[pp. 324-325]{CS}. Apart from the non-embeddability
property, Rossi spheres provide counterexamples in conformal pseudohermitian
geometry. In relation to the problem of existence of minimizers for the CR
Yamabe problem, each Rossi sphere $(J_{(s)},$ $\hat{\theta}),$ $s$ $\neq $ $%
0,$ has negative pseudohermitian mass, as defined in \cite{CMY}, for $s$
close to $0$ while the infimum of the CR Sobolev quotient coincides with the
one for the standard $3$-sphere ($s$ $=$ $0$), but is not attained \cite{CMY}%
. The notion of pseudo-Einstein contact form plays an important role in CR
geometry. Geometrically it is characterized by a volume-normalization
condition while analytically in dimension 3 it relates $R$ to $A_{11}$ in
their first covariant derivatives as follows:%
\begin{equation}
R_{,1}=iA_{11,\bar{1}}  \label{pE-1}
\end{equation}

\noindent (cf. (\ref{pE})). Equation (\ref{pE-1}) is useful in simplifying
the expressions involving $R$ and $A_{11}.$ Among others, one can equate the
Burns-Epstein invariant to the total $Q^{\prime }$-curvature (up to a
negative constant) \cite[Theorem 1.2 on pp. 290-291]{CY}.

In this paper we exhibit a functional whose critical points (in both $J$ and 
$\theta $) characterize homogeneous pseudohermitian manifolds among
pseudo-Einstein ones. Precisely define the following energy functional%
\begin{equation}
E(J,\theta ):=\int_{M}(R^{2}-|A|_{J,\theta }^{2})\theta \wedge d\theta .
\label{En}
\end{equation}%
\noindent Our main result is as follows.

\begin{theorem}
\label{main-H} On a closed (i.e. compact with no boundary) contact
3-manifold $(M,\xi ),$ suppose that $(J,\theta )$ is a pseudo-Einstein
critical point of (\ref{En}). Then the universal cover $\tilde{M}$ of $M$
with the structure naturally inherited from $(\xi ,J,\theta )$, still
denoted by the same notation, is homogeneous as a pseudohermitian 3-manifold.
\end{theorem}

In the Riemannian case, a variational characterization of space forms was
given in \cite{GV}. It is well known that homogeneous CR 3-manifolds have
been classified by Cartan (\cite{Cartan}; see also \cite{BJ}). In
particular, when $(M,\xi )$ $=$ $(S^{3},\hat{\xi}),$ we have the following
characterization for Rossi spheres.

\begin{corollary}
\label{main} On the standard contact 3-sphere $(S^{3},\hat{\xi}),$ it holds
that $(J,\theta )$ is a pseudo-Einstein critical point of (\ref{En}) if and
only if $J$ is isomorphic to a Rossi sphere $J_{(s)}$ for some unique $s$ $%
\leq $ $0$ and $\theta $ is a constant multiple of $\hat{\theta}$.
\end{corollary}

In Section \ref{Sec2}, after a short review of Rossi spheres in Subsection %
\ref{Subsection1-1} and basic variation formulas in Subsection \ref%
{Subsection1-2}, we derive the Euler-Lagrange equation (\ref{EL}) for the
critical points of the energy functional $E(J,\theta ).$ In Section \ref%
{Sec3}, from (\ref{EL}) and $(J,\theta )$ being pseudo-Einstein it follows
that $(M,\xi ,J,\theta )$ (or $(M,J,\theta )$ with $\xi $ omitted) is
locally sub-symmetric through Proposition \ref{P-pE_Sym} and Lemma \ref%
{L-APT}. In Section \ref{Sec3} we first prove Theorem \ref{main-H}. Then for 
$(M,\xi )$ $=$ $(S^{3},\hat{\xi}),$ we conclude Corollary \ref{main} by a
result in \cite{BJ}. We remark that $\iota :$ $(z^{1},z^{2})$ $\rightarrow $ 
$(iz^{1},z^{2})$ is a pseudohermitian isomorphism between $(J_{(s)},$ $\hat{%
\theta})$ and $(J_{(-s)},$ $\hat{\theta})$ \cite{CMY1}. In Section \ref{Sec4}%
, for independent interest we compute the second variation of the energy
functional $E(J,\theta )$ at a critical point $(\hat{J},\hat{\theta}).$

\begin{theorem}
\label{T-3} With the notation above, we obtain the second variation formulas 
$\delta _{J}^{2}E(\hat{J},\hat{\theta}),$ $\delta _{\theta }\delta _{J}E(%
\hat{J},\hat{\theta})$ ($=$ $\delta _{J}\delta _{\theta }E(\hat{J},\hat{%
\theta})$) and $\delta _{\theta }^{2}E(\hat{J},\hat{\theta})$ as in (\ref{S5}%
), (\ref{S7}) and (\ref{S8}) respectively.
\end{theorem}

We remark that there is no characterization in general on sign of the second
variation. In fact, using for examples the Fourier decompositions for the
variations of $J$ and $\theta $ as in \cite{ACMY}, it is possible to find
deformations along which the second differential is either positive or
negative at the standard $S^{3}.$ See (\ref{A8}) and (\ref{A12}) for
examples in the Appendix.

\bigskip

{\large Acknowledgments. }J.-H. C. would like to thank the Ministry of
Science and Technology of Taiwan for the support: grant no. MOST
110-2115-M-001-015 and the National Center for Theoretical Sciences for the
constant support. A.M. is supported by the project \emph{Geometric Problems
with Loss of Compactness} from Scuola Normale Superiore. He is also a member
of GNAMPA as part of INdAM. P. Y. acknowledges support from the NSF for the
grant DMS 1509505.

\section{Rossi spheres, variations and energy functional\label{Sec2}}

\subsection{Rossi spheres\label{Subsection1-1}}

In this subsection we give a short introduction to Rossi spheres. See \cite[%
Section 2.2]{CMY1} for more details. The standard contact form $\hat{\theta}$
on $S^{3}$ $:=$ $\{(z^{1},z^{2})$ $\in $ $C^{2}\text{,}$ $%
|z^{1}|^{2}+|z^{2}|^{2}$ $=$ $1\}$ reads as%
\begin{eqnarray}
\hat{\theta} &=&i(\bar{\partial}-\partial )(|z^{1}|^{2}+|z^{2}|^{2})
\label{Rz} \\
&=&i\sum_{k=1}^{2}(z^{k}dz^{\bar{k}}-z^{\bar{k}}dz^{k}).  \notag
\end{eqnarray}

\noindent Note that $\hat{\theta}$ is $SU(2)$-invariant, where $SU(2)$ acts
on $\mathbb{C}^{2}$ in the canonical way. Dual to $\theta ^{1}$ $=$ $%
z^{2}dz^{1}$ $-$ $z^{1}dz^{2}\text{,}$ we have%
\begin{equation*}
Z_{1}=z^{\bar{2}}\dfrac{\partial }{\partial z^{1}}-z^{\bar{1}}\dfrac{%
\partial }{\partial z^{2}}\text{.}
\end{equation*}

\noindent Note that $Z_{1}$ (resp. $Z_{\bar{1}})$ is also $SU(2)$-invariant.

Consider the deformation of CR structures described by giving type $(0,1)$
vectors as follows: 
\begin{equation*}
Z_{\bar{1}(s)}=Z_{\bar{1}}+\frac{iE_{\bar{1}\bar{1}}s}{\sqrt{1+|E_{11}s|^{2}}%
}Z_{1}
\end{equation*}

\noindent where $E_{\bar{1}\bar{1}}$ is a deformation tensor associated to
the CR structure $J$ (cf. Subsection \ref{Subsection1-2}). Note that we use
the notation $E_{\bar{1}\bar{1}}$ instead of $E_{\bar{1}}$$^{1}$ for
convenience/simplicity; for a unitary frame/coframe they are equal. The
derivative of $Z_{\bar{1}(s)}$ in $s$ reads as%
\begin{equation}
\dot{Z}_{1(s)}=\frac{-iE_{11}Z_{\bar{1}}}{(1+|E_{11}|^{2}s^{2})^{3/2}}\text{.%
}  \label{R0}
\end{equation}

\noindent We express $Z_{1}$ and $Z_{\bar{1}}$ in terms of $Z_{1(s)}$ and $%
Z_{\bar{1}(s)}$ as follows:%
\begin{eqnarray}
Z_{1} &=&iE_{11}s\sqrt{1+|E_{11}s|^{2}}Z_{\bar{1}%
(s)}+(1+|E_{11}s|^{2})Z_{1(s)}\text{,}  \label{R1} \\
Z_{\bar{1}} &=&(-i)E_{\bar{1}\bar{1}}s\sqrt{1+|E_{11}s|^{2}}%
Z_{1(s)}+(1+|E_{11}s|^{2})Z_{\bar{1}(s)}\text{.}  \notag
\end{eqnarray}

\noindent Substituting the second equality of (\ref{R1}) into (\ref{R0})
gives%
\begin{equation}
\dot{Z}_{1(s)}=\frac{-|E_{11}|^{2}s\sqrt{1+|E_{11}s|^{2}}%
Z_{1(s)}-iE_{11}(1+|E_{11}s|^{2})Z_{\bar{1}(s)}}{(1+|E_{11}|^{2}s^{2})^{3/2}}%
.  \label{R3}
\end{equation}

\noindent Differentiating $J_{(s)}Z_{1(s)}=iZ_{1(s)}$ with respect to $s$ we
get%
\begin{equation}
\dot{J}_{(s)}Z_{1(s)}+J_{(s)}\dot{Z}_{1(s)}=i\dot{Z}_{1(s)}.  \label{R4}
\end{equation}

\noindent Substituting (\ref{R3}) into (\ref{R4}) and writing $\dot{J}%
_{(s)}=2E_{11}^{(s)}\theta _{(s)}^{1}\otimes Z_{\bar{1}(s)}+$ conjugate, we
obtain 
\begin{equation*}
E_{11}^{(s)}=\frac{E_{11}}{\sqrt{1+|E_{11}s|^{2}}}\text{.}
\end{equation*}

\noindent Therefore we have 
\begin{equation*}
\dot{E}_{11}^{(s)}=-E_{11}|E_{11}|^{2}s(1+|E_{11}|^{2}s^{2})^{-3/2}.
\end{equation*}

For Rossi spheres, we take 
\begin{equation*}
E_{11}=i\text{.}
\end{equation*}%
\noindent Observe that%
\begin{eqnarray}
Z_{1(s)} &=&Z_{1}+\dfrac{s}{\sqrt{1+s^{2}}}Z_{\bar{1}},  \label{R6-1} \\
Z_{\bar{1}(s)} &=&Z_{\bar{1}}+\dfrac{s}{\sqrt{1+s^{2}}}Z_{1}  \notag
\end{eqnarray}%
\noindent are $SU(2)$-invariant since both $Z_{1}$ and $Z_{\bar{1}}$ are $%
SU(2)$-invariant. Dual to (\ref{R6-1}) we have%
\begin{eqnarray}
\theta _{(s)}^{1} &=&(1+s^{2})\theta ^{1}-s\sqrt{1+s^{2}}\theta ^{\bar{1}}%
\text{,}  \label{R6-2} \\
\theta _{(s)}^{\bar{1}} &=&(1+s^{2})\theta ^{\bar{1}}-s\sqrt{1+s^{2}}\theta
^{1}\text{.}  \notag
\end{eqnarray}

\noindent Compute%
\begin{equation}
i\theta _{(s)}^{1}\wedge \theta _{(s)}^{\bar{1}}=(1+s^{2})i\theta ^{1}\wedge
\theta ^{\bar{1}}=(1+s^{2})d\theta ,  \label{R7}
\end{equation}

\noindent where $d\theta $ $=$ $i\theta ^{1}\wedge \theta ^{\bar{1}}\text{,}$
i.e., $h_{1\bar{1}}$ $=$ $1.$ So from (\ref{R7}) it follows that 
\begin{equation}
h_{1\bar{1}}^{(s)}=\frac{1}{1+s^{2}}\quad \text{ and }\quad h_{(s)}^{1\bar{1}%
}:=(h_{1\bar{1}}^{(s)})^{-1}=1+s^{2}\text{.}  \label{h11}
\end{equation}

\noindent Suppose that the Webster curvature $R$ of $(J,\theta )$ is a
positive constant $R_{0}\text{ (see \cite{Lee} and \cite{Lee2} for basic
pseudohermitian geometry).}$ Then we should take $\omega _{1}^{1}$ $=$ $%
-iR_{0}\theta $ in the structure equation, so that $d\omega _{1}^{1}$ $=$ $%
R_{0}\theta ^{1}\wedge \theta ^{\bar{1}}\text{.}$ For $\theta $ $=$ $\hat{%
\theta}$ (the standard contact form on $S^{3}\text{,}$ see (\ref{Rz})), $%
R_{0}$ $=$ $1$ while $R$ $=$ $R_{0}$ for $\theta $ $=$ $\hat{\theta}/R_{0}%
\text{.}$ Hence 
\begin{equation}
\omega _{1}^{1}=-iR_{0}\theta =-i\hat{\theta}=-2(z^{\bar{1}}dz^{1}+z^{\bar{2}%
}dz^{2}),  \label{w11}
\end{equation}
\noindent where we have used that $z^{\bar{1}}dz^{1}$ $+$ $z^{\bar{2}}dz^{2}$
$+$ conjugate $=$ $0$ on $S^{3}\text{.}$ We can then determine, from the
structure equation for $(J_{(s)},\theta )\text{,}$ that%
\begin{eqnarray}
\omega _{1(s)}^{1} &=&(-i)(1+2s^{2})R_{0}\theta ,  \label{BK} \\
h_{(s)}^{1\bar{1}}A_{\bar{1}\bar{1}(s)} &=&2is\sqrt{1+s^{2}}R_{0},  \notag \\
R_{(s)} &=&(1+2s^{2})R_{0}.  \notag
\end{eqnarray}

\noindent It follows that $(J_{(s)},\theta )$ is pseudo-Einstein (see (\ref%
{pE})) since $R_{(s),1(s)}$ $=$ $0$ $=$ $A_{11(s),\bar{1}(s)}.$ For a
pseudohermitian structure $(J,\theta )$ we recall that the sublaplacian $%
\Delta _{b}$ acting on a function $u$ reads as%
\begin{eqnarray*}
\Delta _{b}u &=&h^{1\bar{1}}(u_{,1\bar{1}}+u_{,\bar{1}1}) \\
&=&u_{,1\bar{1}}+u_{,\bar{1}1}\text{ for a unitary frame (so }h^{1\bar{1}%
}=1).
\end{eqnarray*}

\subsection{Basic formulas for variations in $J$ and $\protect\theta \label%
{Subsection1-2}$}

In this subsection we provide basic formulas for variations of a
pseudohermitian manifold $(M,J,\theta )$. Write the variation of $J$ as $%
\delta J=2E,$ $E=E_{11}\theta ^{1}\otimes Z_{\bar{1}}+$conjugate (for a
unitary coframe/frame $\theta ^{1}/Z_{1}$), $\delta \theta =2h\theta .$
Recall (\cite{CL}, \cite{CMY} or \cite{ACMY}) that 
\begin{eqnarray}
\delta _{J}R &=&(iE_{11,\bar{1}\bar{1}}-A_{\bar{1}\bar{1}}E_{11})+\text{%
conjugate,}  \label{B1} \\
\delta _{J}A_{11} &=&iE_{11,0,}  \notag \\
\delta _{\theta }R &=&-2hR-4\triangle _{b}h,  \notag \\
\delta _{\theta }A_{11} &=&-2hA_{11}+2ih_{,11}.  \notag
\end{eqnarray}

\noindent Define the energy functional $E(J,\theta )$ for a pseudohermitian
manifold $(M,J,\theta )$ by%
\begin{equation*}
E(J,\theta ):=\int_{M}(R^{2}-|A|_{J,\theta }^{2})\theta \wedge d\theta
\end{equation*}

\noindent (cf. (\ref{En})). We have the following first variation formulas
for $E(J,\theta ).$

\begin{proposition}
\label{P-1stV} Suppose $(M,J,\theta )$ is a closed (compact with no
boundary) pseudohermitian $3$-manifold. Then we have, for a unitary frame $%
Z_{1}$ (and coframe $\theta ^{1}$)%
\begin{equation}
\delta _{J}E(J,\theta )=\int_{M}(2iR_{,\bar{1}\bar{1}}-2RA_{\bar{1}\bar{1}%
}+iA_{\bar{1}\bar{1},0})E_{11}+\QTR{up}{conjugate},  \label{1stV}
\end{equation}%
\begin{equation}
\delta _{\theta }E(J,\theta )=\int_{M}\{-8\Delta _{b}R-2i(A_{\bar{1}\bar{1}%
,11}-A_{11,\bar{1}\bar{1}})\}h\theta \wedge d\theta ,  \label{1stV2}
\end{equation}%
where $\delta J=2E,$ $E=E_{11}\theta ^{1}\otimes Z_{\bar{1}}+\QTR{up}{%
conjugate}$ and $\delta \theta =2h\theta .$ So, the Euler-Lagrange equation
for $(J,\theta )$ reads as%
\begin{equation}
\QATOPD\{ . {R_{,11}-\frac{i}{2}(A_{11,1\bar{1}}-A_{11,\bar{1}1})=0\text{,}%
}{-4\Delta _{b}R-i(A_{\bar{1}\bar{1},11}-A_{11,\bar{1}\bar{1}})=0.}
\label{EL}
\end{equation}
\end{proposition}

\begin{proof}
Making use of the first two formulas in (\ref{B1}) and the integration by
parts, we get (\ref{1stV}) while using the last two formulas in (\ref{B1})
and integrating by parts gives (\ref{1stV2}). Observe that $-iRA_{11}+\frac{1%
}{2}A_{11,0}$ equals $-\frac{i}{2}(A_{11,1\bar{1}}-A_{11,\bar{1}1})$ by the
commutation relation $iA_{11,0}$ + $2RA_{11}$ $=$ $A_{11,1\bar{1}}-A_{11,%
\bar{1}1}$ (\cite{Lee2})$.$ Together with (\ref{1stV}) and (\ref{1stV2}), we
conclude the proof of (\ref{EL}).
\end{proof}

A direct computation shows that each Rossi sphere $(S^{3},J_{(s)},\theta )$
is a solution to (\ref{EL}) by (\ref{BK}) and noting that $A_{11,0}$ $=$ $%
TA_{11}-2\omega _{1}^{1}(T)A_{11}.$ We notice that the coefficient of $%
|A|_{J,\theta }^{2}$ in the integrand of $E(J,\theta )$ is different from
that in the integrand of the following energy functional 
\begin{equation*}
\int (R^{2}-4|A|_{J,\theta }^{2})\theta \wedge d\theta ,
\end{equation*}

\noindent which is known to be the total $Q^{\prime }$-curvature (for
pseudo-Einstein ($J,\theta )),$ whose critical points are spherical, see 
\cite{CY}.

In order to compute the second variation, we need the formulas for $\delta
_{J}R_{,\bar{1}\bar{1}}=\overline{\delta _{J}R_{,11}},$ $\delta _{J}A_{\bar{1%
}\bar{1},0}$ $=$ $\overline{\delta _{J}A_{11,0}},$ $\delta _{\theta }R_{,11}$%
, $\delta _{\theta }A_{11,0},$ $\delta _{\theta }(\Delta _{b}R)$ and $\delta
_{\theta }A_{\bar{1}\bar{1},11}.$ We compute these quantities at a
pseudo-Einstein critical point $(\hat{J},\hat{\theta})$ of $E(J,\theta ),$
where $\hat{R}$ $=$ const and $\hat{A}_{11,1}$ $=$ $\hat{A}_{11,\bar{1}}$ $=$
$0.$ First we obtain

\begin{eqnarray}
\delta _{J}R_{,11} &=&(\delta _{J}R)_{,11}\text{ at }(\hat{J},\hat{\theta})
\label{B3} \\
&&\overset{(\ref{B1})}{=}[iE_{11,\bar{1}\bar{1}}-\hat{A}_{\bar{1}\bar{1}%
}E_{11}-iE_{\bar{1}\bar{1},11}-\hat{A}_{11}E_{\bar{1}\bar{1}}]_{,11}.  \notag
\end{eqnarray}

\noindent Recall that $A_{\bar{1}\bar{1},0}$ $=$ $TA_{11}-2\omega
_{1}^{1}(T)A_{11}.$ So, using $\delta _{J}\omega _{1}^{1}$ $=$ $i(A_{11}E_{%
\bar{1}\bar{1}}+A_{\bar{1}\bar{1}}E_{11})\theta $ $-$ $i(E_{11,\bar{1}%
}\theta ^{1}+E_{\bar{1}\bar{1},1}\theta ^{\bar{1}})$ we have%
\begin{eqnarray}
\delta _{J}A_{11,0} &=&(\delta _{J}A_{11})_{,0}-2(\delta \omega
_{1}^{1})(T)A_{11}  \label{B4} \\
&=&iE_{11,00}-2i(A_{11}E_{\bar{1}\bar{1}}+A_{\bar{1}\bar{1}}E_{11})A_{11}. 
\notag
\end{eqnarray}

For variations in $\theta $ we have the following basic formulas:%
\begin{eqnarray}
\delta \theta &=&2h\theta ,  \label{B5} \\
\delta _{\theta }Z_{1} &=&-hZ_{1,}  \notag \\
\delta _{\theta }T &=&-2hT+2ih_{,1}Z_{\bar{1}}-2ih_{,\bar{1}}Z_{1}.  \notag
\end{eqnarray}

\noindent Besides $\delta _{\theta }R,$ $\delta _{\theta }A_{11}$ in (\ref%
{B1}) we also have%
\begin{equation}
\delta _{\theta }\omega _{1}^{1}=3h_{,1}\theta ^{1}-3h_{,\bar{1}}\theta ^{%
\bar{1}}+i(\triangle _{b}h)\theta .  \label{B6}
\end{equation}

\noindent Writing $R_{,11}$ $=$ $Z_{1}(Z_{1}R)-$ $\omega
_{1}^{1}(Z_{1})Z_{1}R,$ we compute $\delta _{\theta }(R_{,11})$ as follows:%
\begin{eqnarray}
\delta _{\theta }(R_{,11}) &=&\delta _{\theta }Z_{1}(Z_{1}R)+Z_{1}(\delta
_{\theta }Z_{1})R+Z_{1}Z_{1}(\delta _{\theta }R)  \label{B7} \\
&&-(\delta _{\theta }\omega _{1}^{1})(Z_{1})Z_{1}R-\omega _{1}^{1}(\delta
_{\theta }Z_{1})Z_{1}R  \notag \\
&&-\omega _{1}^{1}(Z_{1})(\delta _{\theta }Z_{1})R-\omega
_{1}^{1}(Z_{1})Z_{1}(\delta _{\theta }R)  \notag \\
&=&-4hR_{,11}-8h_{,1}R_{,1}-2h_{,11}R-4(\triangle _{b}h)_{,11}  \notag
\end{eqnarray}

\noindent by (\ref{B5}), (\ref{B6}) and (\ref{B1}). Similarly writing $%
A_{11,0}$ $=$ $TA_{11}$ $-$ $2\omega _{1}^{1}(T)A_{11},$ we compute 
\begin{eqnarray}
\delta _{\theta }(A_{11,0}) &=&(\delta _{\theta }T)A_{11}+T(\delta _{\theta
}A_{11})-2(\delta _{\theta }\omega _{1}^{1})(T)A_{11}  \label{B8} \\
&&-2\omega _{1}^{1}(\delta _{\theta }T)A_{11}-2\omega _{1}^{1}(T)\delta
_{\theta }A_{11}  \notag \\
&=&-4hA_{11,0}+2ih_{,1}A_{11,\bar{1}}-2ih_{,\bar{1}}A_{11,1}  \notag \\
&&-2h_{,0}A_{11}-2i(\triangle _{b}h)A_{11}+2ih_{,110}  \notag
\end{eqnarray}

\noindent by (\ref{B5}), (\ref{B6}) and (\ref{B1}) again. For the second
variation in $\theta $ we also need to compute%
\begin{eqnarray}
&&\delta _{\theta }(\Delta _{b}R)  \label{B9} \\
&=&\Delta _{b}(\delta _{\theta }R)\text{ at the critical points where }R=%
\hat{R}=\text{const.}  \notag \\
&=&-2(\Delta _{b}h)\hat{R}-4\Delta _{b}^{2}h.  \notag
\end{eqnarray}

\noindent and using (\ref{B5})$,$ (\ref{B1}) and (\ref{B6}), we compute 
\begin{eqnarray}
\delta _{\theta }A_{\bar{1}\bar{1},1} &=&(\delta _{\theta }Z_{1})A_{\bar{1}%
\bar{1}}+(\delta _{\theta }A_{\bar{1}\bar{1}})_{,1}  \label{B10} \\
&&+2(\delta _{\theta }\omega _{1}^{1})(Z_{1})A_{\bar{1}\bar{1}}+2\omega
_{1}^{1}(\delta _{\theta }Z_{1})A_{\bar{1}\bar{1}}  \notag \\
&=&-hA_{\bar{1}\bar{1},1}-2h_{,1}A_{\bar{1}\bar{1}}-2hA_{\bar{1}\bar{1}%
,1}-2ih_{,\bar{1}\bar{1}1}+6h_{,1}A_{\bar{1}\bar{1}}  \notag \\
&=&-3hA_{\bar{1}\bar{1},1}+4h_{,1}A_{\bar{1}\bar{1}}-2ih_{,\bar{1}\bar{1}1}.
\notag
\end{eqnarray}

\noindent We then compute, at $(\hat{J},\hat{\theta})$ where $\hat{A}_{\bar{1%
}\bar{1},1}=0,$%
\begin{eqnarray}
\delta _{\theta }A_{\bar{1}\bar{1},11} &=&(\delta _{\theta }A_{\bar{1}\bar{1}%
,1})_{,1}  \label{B11} \\
&=&4h_{,11}\hat{A}_{\bar{1}\bar{1}}-2ih_{,\bar{1}\bar{1}11}.  \notag
\end{eqnarray}

\section{Proofs of Theorem \protect\ref{main-H} and Corollary \protect\ref%
{main}\label{Sec3}}

\begin{proposition}
\label{P-pE_Sym} Suppose $(M,\xi )$ is a closed (compact with no boundary)
contact $3$-manifold. Suppose $(J,\theta )$ on $(M,\xi )$ is pseudo-Einstein
and a solution to (\ref{EL}). Then $R$ $=$ \textup{const.} and $A_{11,1}$ $=$
$0,$ $A_{11,\bar{1}}$ $=$ $0.$
\end{proposition}

\begin{proof}
By the condition of $(J,\theta )$ being pseudo-Einstein, we have 
\begin{equation}
R_{,1}=iA_{11,\bar{1}}.  \label{pE}
\end{equation}

\noindent Substituting (\ref{pE}) into the second equation of (\ref{EL})
gives $3\Delta _{b}R=0$ by noting that $\Delta _{b}R$ $=$ $R_{,1\bar{1}}+R_{,%
\bar{1}1}$ ($h^{1\bar{1}}$ $=$ $h_{1\bar{1}}$ $=$ $1$)$.$ It follows that $R$
$=$ const. since $M$ is closed. We now multiply the second equation of (\ref%
{EL}) by $R$ and integrate over $M$ with respect to the volume form $\theta
\wedge d\theta .$ After integrating by parts we obtain%
\begin{equation}
\int_{M}(-4|\nabla _{b}R|^{2}+iR_{,11}A_{\bar{1}\bar{1}}-iR_{,\bar{1}\bar{1}%
}A_{11})\theta \wedge d\theta =0,  \label{Int}
\end{equation}

\noindent where $|\nabla _{b}R|^{2}$ $:=$ $2h^{1\bar{1}}R_{,1}R_{,\bar{1}}$ $%
=$ $2R_{,1}R_{,\bar{1}}.$ From the first equation of (\ref{EL}) and the
commutation relation 
\begin{equation*}
iA_{11.0}=A_{11,1\bar{1}}-A_{11,\bar{1}1}-2RA_{11},
\end{equation*}

\noindent it follows that%
\begin{eqnarray}
R_{,11} &=&iRA_{11}-\frac{1}{2}A_{11,0}  \label{R2} \\
&=&\frac{1}{2}i(A_{11,1\bar{1}}-A_{11,\bar{1}1}).  \notag
\end{eqnarray}

\noindent Substituting (\ref{R2}) into (\ref{Int}) gives%
\begin{equation}
\int_{M}(-4|\nabla _{b}R|^{2}+|A_{11,1}|^{2}-|A_{11,\bar{1}}|^{2})\theta
\wedge d\theta =0  \label{Int-1}
\end{equation}

\noindent by integrating by parts. Now making use of $R$ $\equiv $ const.
and (\ref{pE}) (so $A_{11,\bar{1}}$ $=$ $0$) in (\ref{Int-1}), we get $%
\int_{M}|A_{11,1}|^{2}\theta \wedge d\theta =0$ and hence $A_{11,\bar{1}}$ $%
= $ $0.$ We have completed the proof.
\end{proof}

Let $\tau $ denote the torsion tensor of the pseudohermitian connection $%
\nabla ,$ i.e.%
\begin{equation*}
\tau (U,V):=\nabla _{U}V-\nabla _{V}U-[U,V]
\end{equation*}

\noindent for any tangent vector fields $U,$ $V$. Recall that the Reeb
vector field $T$ is the unique vector field such that $\theta (T)$ $=$ $1$
and $d\theta (T,\cdot )$ $=$ $0.$ It is not hard to see from the formulas
for Lie brackets in \cite[page 418]{Lee} that for $Y,W$ $\in $ $\xi $ 
\begin{equation}
\tau (Y,W)=d\theta (Y,W)T,  \label{3-1}
\end{equation}

\noindent and%
\begin{equation}
U(g_{J,\theta }(Y,W))=g_{J,\theta }(\nabla _{U}Y,W)+g_{J,\theta }(Y,\nabla
_{U}W),  \label{3-1a}
\end{equation}

\noindent where $g_{J,\theta }$ is the Levi metric defined by $g_{J,\theta
}(Y,W)$ $:=$ $d\theta (Y,JW)$ and $U$ is any tangent vector.

\begin{lemma}
\label{L-APT} On a pseudohermitian 3-manifold $(M,\xi ,J,\theta ),$ suppose $%
R$ $\equiv$ \textup{const.} and $A_{11,1}$ $=$ $0,$ $A_{11,\bar{1}}$ $=$ $0.$
Then $\nabla _{X}R$ $=$ $0$ and $\nabla _{X}\tau $ $=$ $0$ for any $X\in \xi
.$
\end{lemma}

\begin{proof}
It is clear that $R$ $=$ const implies $\nabla _{X}R$ $=$ $X(R)=0.$ To prove
that $\nabla _{X}\tau $ $=$ $0$ for any $X\in \xi ,$ it is enough to show $%
(\nabla _{X}\tau )(Y,W)$ $=$ $0$ for $Y,W$ $\in $ $\xi $, $(\nabla _{X}\tau
)(Y,T)$ $=$ $(\nabla _{X}\tau )(T,Y)$ $=$ $0$ and $(\nabla _{X}\tau )(T,T)$ $%
=$ $0.$ We compute%
\begin{eqnarray}
&&\nabla _{X}(\tau (Y,W)\overset{(\ref{3-1})}{=}X(d\theta (Y,W))T+d\theta
(Y,W)\nabla _{X}T  \label{3-2} \\
&=&X(d\theta (Y,W))T\text{ \ (}\nabla T=0\text{ by \cite[(4.5) on page 418]%
{Lee}).}  \notag
\end{eqnarray}

\noindent It follows from (\ref{3-2}) and (\ref{3-1}) that

\begin{eqnarray*}
(\nabla _{X}\tau )(Y,W) &=&\nabla _{X}(\tau (Y,W))-\tau (\nabla
_{X}Y,W)-\tau (Y,\nabla _{X}W) \\
&=&\{X(d\theta (Y,W))-d\theta (\nabla _{X}Y,W)-d\theta (Y,\nabla _{X}W)\}T \\
&=&0.
\end{eqnarray*}

\noindent Here we have used that $\nabla $ preserves $\xi $ \cite[(4.5)]{Lee}
and applied (\ref{3-1a}) with $W$ replaced by $-JW$ and the property that $%
\nabla \circ J=J\circ \nabla $ (see \cite[Proposition 3.1 2)]{Tan} for
instance). We next compute%
\begin{eqnarray}
(\nabla _{X}\tau )(Y,T) &=&\nabla _{X}(\tau (Y,T))-\tau (\nabla
_{X}Y,T)-\tau (Y,\nabla _{X}T)  \label{3-3} \\
&=&\nabla _{X}(\tau (Y,T))-\tau (\nabla _{X}Y,T)\text{ (since }\nabla T=0). 
\notag
\end{eqnarray}

Define the tensor $A$ $:=$ $A_{1}^{\bar{1}}\theta ^{1}\otimes Z_{\bar{1}}+A_{%
\bar{1}}^{1}\theta ^{\bar{1}}\otimes Z_{1}$ where $A_{1}^{\bar{1}}$ $=$ $h^{1%
\bar{1}}A_{11}$ and $A_{\bar{1}}^{1}$ is the complex conjugate of $A_{1}^{%
\bar{1}}.$ Observe that (extending the defining domain of $\tau $ to complex
tangent vectors by complex linearity) 
\begin{eqnarray}
\tau (Z_{1},T) &=&\nabla _{Z_{1}}T-\nabla _{T}Z_{1}-[Z_{1},T]  \label{3-4} \\
&=&0-\omega _{1}^{1}(T)Z_{1}-(A_{1}^{\bar{1}}Z_{\bar{1}}-\omega
_{1}^{1}(T)Z_{1})  \notag \\
&=&-A_{1}^{\bar{1}}Z_{\bar{1}}=-A(Z_{1})  \notag
\end{eqnarray}

\noindent by \cite[page 418]{Lee}. In fact we also have $\tau (fZ_{1},T)$ $=$
$-A(fZ_{1})$ for any complex function $f.$ So it holds that%
\begin{equation}
\tau (\nabla _{X}Z_{1},T)=-A(\nabla _{X}Z_{1}).  \label{3-5}
\end{equation}%
\noindent It follows from (\ref{3-4}) and (\ref{3-5}) that%
\begin{eqnarray}
&&\nabla _{X}(\tau (Z_{1},T))-\tau (\nabla _{X}Z_{1},T)  \label{3-6} \\
&=&-\nabla _{X}(A(Z_{1})+A(\nabla _{X}Z_{1})  \notag \\
&=&-(\nabla _{X}A)(Z_{1})=0  \notag
\end{eqnarray}

\noindent by the condition $A_{11,1}$ $=$ $0,$ $A_{11,\bar{1}}$ $=$ $0$ due
to Proposition \ref{P-pE_Sym}$.$ By taking complex conjugation, (\ref{3-6})
also holds for $Z_{\bar{1}}$ replacing $Z_{1}.$ So the RHS of (\ref{3-3})
vanishes. We have shown $(\nabla _{X}\tau )(Y,T)=0.$ Noting that $(\nabla
_{X}\tau )(T,Y)$ $=$ $-(\nabla _{X}\tau )(Y,T)=0.$ Clearly $(\nabla _{X}\tau
)(T,T)$ $=$ $0$ since $\tau $ is skew-symmetric. We have completed the proof.
\end{proof}

We call a pseudohermitian automorphism $\phi $ of $(M,\xi ,J,\theta )$ a 
\emph{sub-symmetry} at a point $x$ if $\phi (x)$ $=$ $x$ and $\phi _{\ast
}|_{\xi _{x}}=-1$ ($\phi _{\ast }T_{x}$ $=$ $T_{x}$ hence the Reeb orbit
through $x$ is fixed by $\phi ).$ A local sub-symmetry at a point $x$ means
that $\phi $ is only defined in a neighborhood of $x.$

\bigskip

\begin{pfn}
\emph{of Theorem \ref{main-H}}. Suppose that $(J,\theta )$ is a
pseudo-Einstein critical point of (\ref{En}). By Proposition \ref{P-1stV} $%
(J,\theta )$ satisfies the system (\ref{EL}) and hence from Proposition \ref%
{P-pE_Sym} it follows that%
\begin{equation*}
R=\text{const., }A_{11,1}=0,A_{11,\bar{1}}=0.
\end{equation*}

\noindent By Lemma \ref{L-APT} we obtain that the curvature $R$ and the
torsion (tensor) $\tau $ are parallel along the (horizontal) direction of
any contact vector. From the proof of \cite[Theorem 2.1]{Fal} (noting that
the Levi metric $g_{J,\theta }$ plays the role of the metric in the setting
of \cite{Fal}), for each point $x$ we can find a local sub-symmetry $\phi
_{x}$ such that $\phi _{x}(x)$ $=$ $x,$ $\phi _{x}^{2}$ $=$ $Id.$ Lift $\phi
_{x}$ to a local sub-symmetry $\tilde{\phi}_{\tilde{x}}$ on $\tilde{M},$ the
universal cover of $M,$ where $\tilde{x}$ $\in $ $\tilde{M}$ is a lift of $%
x,.$i.e. $\pi (\tilde{x})$ $=$ $x,$ $\pi $ $:$ $\tilde{M}$ $\rightarrow $ $M$
is the natural projection. Since $\tilde{M}$ is simply connected, we can
extend $\phi _{x}$ uniquely to a global pseudohermitian automorphism using
the parabolic exponential map in \cite[page 309]{JL} by a similar argument
in \cite[pp. 252-255]{KN} for extending an affine map. Observe that the
fixed point set of $\tilde{\phi}_{\tilde{x}}$ is a Reeb orbit $F_{\tilde{x}}$
in $\tilde{M},$ $\{F_{\tilde{x}}\}_{\tilde{x}\in \tilde{M}}$ foliate $\tilde{%
M}$ and $\{\tilde{\phi}_{\tilde{x}}\}_{\tilde{x}\in \tilde{M}}$ permutes the
Reeb orbits $\{F_{\tilde{x}}\}_{\tilde{x}\in \tilde{M}},$ since all $\tilde{%
\phi}_{\tilde{x}}$'s are pseudohermitian automorphisms. The sub-symmetry $%
\tilde{\phi}_{\tilde{x}}$ has the following properties:%
\begin{eqnarray*}
\tilde{\phi}_{\tilde{x}}(\tilde{x}) &=&\tilde{x},\text{ (}\tilde{\phi}_{%
\tilde{x}})_{\ast }|_{\xi _{\tilde{x}}}=-Id,\text{ }\tilde{\phi}_{\tilde{x}%
}|_{F_{\tilde{x}}}=Id,\text{ } \\
\tilde{\phi}_{\tilde{x}}^{2} &=&Id,\text{ so }\tilde{\phi}_{\tilde{x}}^{-1}=%
\tilde{\phi}_{\tilde{x}}.
\end{eqnarray*}%
Let $\QTR{up}{Aut}_{\psi .h.}(\tilde{M},\xi ,J,\theta )$ denote the group of
all pseudohermitian automorphisms. We claim that%
\begin{equation}
\QTR{up}{Aut}_{\psi .h.}(\tilde{M},\xi ,J,\theta )\text{ acts on }(\tilde{M}%
,\xi ,J,\theta )\text{ transitively.}  \label{hom}
\end{equation}

\noindent Observe that $\tilde{M}$ is complete (meaning that it is complete
as a metric space) by, for instance, \cite[Theorem 7.1 (b)]{Str}. Next,
given $p,$ $q$ $\in $ $\tilde{M},$ we can find a Legendrian (horizontal)
geodesic (with respect to the Levi-metric $g_{J,\theta })$ $\gamma $
connecting $p$ and $q,$ parametrized by the arc length of $g_{J,\theta }$ by 
\cite[Theorem 7.1 (a)]{Str}. Let $m\in \gamma $ be the middle point of the
curve $\gamma .$ It follows that $\tilde{\phi}_{m}$ maps $p$ (resp. $q$) to $%
q$ (resp. $p$)$.$ We have shown (\ref{hom}). That is, $(\tilde{M},\xi
,J,\theta )$ is homogeneous as a pseudohermitian manifold.
\end{pfn}

\begin{remark}
\label{R-3-0} We notice that in \cite{Fal} the authors make the assumption
on homogeneity to classify all possible sub-symmetric spaces through a
Lie-theoretic argument.
\end{remark}

\begin{pfn}
\emph{of Corollary \ref{main}}. By (\ref{BK}) (together with (\ref{h11})) we
verify that a Rossi sphere $(J_{(s)},$ $\hat{\theta})$ is pseudo-Einstein
and satisfies (\ref{EL}) by noting that ($\theta =\hat{\theta})$%
\begin{eqnarray*}
A_{\bar{1}\bar{1}(s),0} &=&\hat{T}A_{\bar{1}\bar{1}(s)}-2\omega _{1(s)}^{1}(%
\hat{T})A_{11(s)} \\
&=&0-2(-i)(1+2s^{2})R_{0}A_{11(s)} \\
&=&2iR_{(s)}A_{11(s)}.
\end{eqnarray*}

\noindent So, $(J_{(s)},$ $\hat{\theta})$ is a pseudo-Einstein critical
point of (\ref{En}) in view of Proposition \ref{P-1stV}. Conversely, by
Theorem \ref{main-H} $(S^{3},\hat{\xi},J,\theta )$ is homogeneous. In
particular, it is a homogeneous CR 3-manifold. According to Cartan \cite[%
page 69]{Cartan}, the CR structure $J$ must be left-invariant on $SU(2)$ (=$%
S^{3}).$ By \cite[Proposition 5.1 (c)]{BJ} we conclude that $J$ is
isomorphic to a Rossi sphere $J_{(s)}$ for a unique $s$ $\leq $ $0$ (by
comparing (\ref{R6-2}) with the coframe taken in the proof of \cite[%
Proposition 5.1]{BJ}, we get the parameter relation: 
\begin{equation*}
\sqrt{t}=\sqrt{1+s^{2}}-s,
\end{equation*}%
\noindent so $t$ $\geq $ $1$ corresponds to $s$ $\leq $ $0,$ where $t$ is
strictly decreasing as a function of $s$). Moreover, that $\theta $ is $%
SU(2) $-invariant implies that $\theta $ is a constant multiple of $\hat{%
\theta}.$ We have thus completed the proof.
\end{pfn}

\begin{remark}
\label{R-3-1} In \cite{BJ} besides Rossi spheres some other examples of
homogeneous CR 3-manifolds are discussed. Let us write down $R$ and $A_{11}$
for two less known examples: $SL_{2}(\mathbb{R})$ and the Euclidean group $%
E_{2}$ $=$ $SO_{2}\rtimes \mathbb{R}^{2}.$ For $SL_{2}(\mathbb{R})$ there
are a family of homogeneous CR structures with parameter $t$ (see \cite[%
Proposition 4.2]{BJ}). With respect to a suitable unitary coframe in the
proof of \cite[Proposition 4.2]{BJ}, we obtain ($t\neq 0,$ $-1$)%
\begin{equation*}
R_{(t)}=-\frac{1+6t+t^{2}}{4|t|(1+t)},\text{ }A_{11(t)}=i\frac{(1-t)^{2}}{%
4|t|(1+t)}.
\end{equation*}%
For $E_{2}$ there is a unique homogeneous CR structure up to $Aut(E_{2})$
(see \cite[Proposition 7.1 (a)]{BJ}). With respect to a suitable unitary
coframe in the proof of \cite[Proposition 7.1 (a)]{BJ}, we easily obtain
that $R$ $=$ $\frac{1}{2}$ and $A_{11}$ $=$ $\frac{i}{2}.$
\end{remark}

\section{Second variation: proof of Theorem \protect\ref{T-3}\label{Sec4}}

Starting from (\ref{1stV}) we compute the second variation in $J,$ at a
critical point $(\hat{J},\hat{\theta})$ where $(2iR_{,\bar{1}\bar{1}}-2RA_{%
\bar{1}\bar{1}}+iA_{\bar{1}\bar{1},0})=0$,%
\begin{eqnarray}
\delta _{J}^{2}E(\hat{J},\hat{\theta}) &=&\int_{M}(2i\delta _{J}R_{,\bar{1}%
\bar{1}}-2(\delta _{J}R)A_{\bar{1}\bar{1}}  \label{S1} \\
&&-2R(\delta _{J}A_{\bar{1}\bar{1}})+i\delta _{J}A_{\bar{1}\bar{1},0})E_{11}+%
\text{\textup{conjugate}},  \notag
\end{eqnarray}

\noindent Applying (\ref{B3}), (\ref{B1}) and (\ref{B4}) to the RHS of (\ref%
{S1}) we obtain%
\begin{eqnarray}
&&\delta _{J}^{2}E(\hat{J},\hat{\theta})  \label{S2} \\
&=&\int_{M}\left\{ 
\begin{array}{c}
(-iE_{\bar{1}\bar{1},11}+iE_{11,\bar{1}\bar{1}}-\hat{A}_{\bar{1}\bar{1}%
}E_{11}-\hat{A}_{11}E_{\bar{1}\bar{1}})_{,\bar{1}\bar{1}} \\ 
-2i(E_{11,\bar{1}\bar{1}}-\hat{A}_{\bar{1}\bar{1}}E_{11})\hat{A}_{\bar{1}%
\bar{1}}+2i(E_{\bar{1}\bar{1},11}-\hat{A}_{11}E_{\bar{1}\bar{1}})\hat{A}_{%
\bar{1}\bar{1}} \\ 
+2i\hat{R}E_{\bar{1}\bar{1},0}+E_{\bar{1}\bar{1},00}-2\hat{A}_{\bar{1}\bar{1}%
}(\hat{A}_{\bar{1}\bar{1}}E_{11}+\hat{A}_{11}E_{\bar{1}\bar{1}})%
\end{array}%
\right\} E_{11}  \notag \\
&&+\text{conjugate}  \notag
\end{eqnarray}

\noindent (volume form $\hat{\theta}\wedge d\hat{\theta}$ omitted). The
\textquotedblleft slice condition" in \cite[p.235]{CL} for $\delta J=2E$
reads as $B_{\hat{J}}E$ $=$ $0,$ i.e.%
\begin{equation}
iE_{11,\bar{1}\bar{1}}-\hat{A}_{11}E_{\bar{1}\bar{1}}=-iE_{\bar{1}\bar{1}%
,11}-\hat{A}_{\bar{1}\bar{1}}E_{11}.  \label{S3}
\end{equation}

\noindent From the commutation relation $E_{\bar{1}\bar{1},1\bar{1}}-E_{\bar{%
1}\bar{1},\bar{1}1}=iE_{\bar{1}\bar{1},0}-2\hat{R}E_{\bar{1}\bar{1}}$ and an
integration by parts$,$ it follows that%
\begin{eqnarray}
&&\int_{M}2i\hat{R}E_{\bar{1}\bar{1},0}E_{11}\hat{\theta}\wedge d\hat{\theta}
\label{S4} \\
&=&\int_{M}\{4\hat{R}^{2}|E_{11}|^{2}-2\hat{R}|E_{\bar{1}\bar{1},1}|^{2}+2%
\hat{R}|E_{11,1}|^{2}\}\hat{\theta}\wedge d\hat{\theta}.  \notag
\end{eqnarray}

\noindent Making use of (\ref{S3}), (\ref{S4}) and integrating by parts
again, we finally reduce (\ref{S2}) to%
\begin{eqnarray}
&&\delta _{J}^{2}E(\hat{J},\hat{\theta})  \label{S5} \\
&=&\int_{M}\left\{ 
\begin{array}{c}
-2|E_{11,0}|^{2}-4\hat{R}|E_{\bar{1}\bar{1},1}|^{2}+4\hat{R}|E_{11,1}|^{2}
\\ 
+(8\hat{R}^{2}-4|\hat{A}_{11}|^{2})|E_{11}|^{2}%
\end{array}%
\right\}  \notag \\
&&+[(2i+2)E_{11,\bar{1}}^{2}-2iE_{11,1}E_{\bar{1}\bar{1},1}+(2i-2)E_{11}^{2}%
\hat{A}_{\bar{1}\bar{1}}]\hat{A}_{\bar{1}\bar{1}}  \notag \\
&&+\text{conjugate of }[\text{ }\cdot \cdot \cdot \text{ }]\hat{A}_{\bar{1}%
\bar{1}}.  \notag
\end{eqnarray}

\noindent (volume form $\hat{\theta}\wedge d\hat{\theta}$ omitted).

Now we are going to compute%
\begin{eqnarray}
\delta _{\theta }\delta _{J}E(\hat{J},\hat{\theta}) &=&\int_{M}(2i\delta
_{\theta }R_{,\bar{1}\bar{1}}-2(\delta _{\theta }R)\hat{A}_{\bar{1}\bar{1}}
\label{S6} \\
&&-2\hat{R}(\delta _{\theta }A_{\bar{1}\bar{1}})+i\delta _{\theta }A_{\bar{1}%
\bar{1},0})E_{11}+\text{\textup{conjugate}},  \notag
\end{eqnarray}

\noindent Substituting (\ref{B7}), (\ref{B1}) and (\ref{B8}) into (\ref{S6})
gives%
\begin{eqnarray}
&&\delta _{\theta }\delta _{J}E(\hat{J},\hat{\theta})  \label{S7} \\
&=&\int_{M}\left\{ 
\begin{array}{c}
(6\triangle _{b}h-2ih_{,0})\hat{A}_{\bar{1}\bar{1}} \\ 
-8i(\triangle _{b}h)_{,\bar{1}\bar{1}}+2h_{,\bar{1}\bar{1}0}%
\end{array}%
\right\} E_{11}+\text{conjugate.}  \notag
\end{eqnarray}

\noindent To compute $\delta _{\theta }^{2}E(\hat{J},\hat{\theta})$ we apply
(\ref{B9}), (\ref{B11}) to the $\delta _{\theta \text{ }}$ of (\ref{1stV2}):%
\begin{equation*}
\delta _{\theta }^{2}E(\hat{J},\hat{\theta})=\int_{M}\{-8\delta _{\theta
}(\Delta _{b}R)-2i\delta _{\theta }(A_{\bar{1}\bar{1},11}-A_{11,\bar{1}\bar{1%
}})\}h
\end{equation*}

\noindent to conclude via integrating by parts that%
\begin{eqnarray}
&&\delta _{\theta }^{2}E(\hat{J},\hat{\theta})  \label{S8} \\
&=&\int_{M}\left\{ 
\begin{array}{c}
-16\hat{R}|\nabla _{b}h|^{2}+8i\hat{A}_{\bar{1}\bar{1}}(h_{,1})^{2}-8i\hat{A}%
_{11}(h_{,\bar{1}})^{2} \\ 
+32(\triangle _{b}h)^{2}-8|h_{,11}|^{2}%
\end{array}%
\right\} \hat{\theta}\wedge d\hat{\theta}.  \notag
\end{eqnarray}

\noindent For Rossi spheres $(S^{3},J_{(s)},\theta )$ with $\theta $ $=$ $%
\hat{\theta}/R_{0}$ (see Subsection \ref{Subsection1-1}) we compute via (\ref%
{BK}) that%
\begin{eqnarray*}
&&R_{(s)}^{2}-|A|_{J_{(s)},\theta }^{2} \\
&=&(1+2s^{2})^{2}R_{0}^{2}-4s^{2}(1+s^{2})R_{0}^{2}=R_{0}^{2}.
\end{eqnarray*}

\noindent Together with $\hat{\theta}\wedge d\hat{\theta}$ $=$ $%
8dv_{S^{3}}^{Eucl}$ (recall (\ref{Rz}) for $\hat{\theta}$ and $%
dv_{S^{3}}^{Eucl}$ denotes the Euclidean volume form of $S^{3}),$ we have%
\begin{eqnarray*}
E(J_{(s)},\theta ) &=&\int_{S^{3}}(R_{(s)}^{2}-|A|_{J_{(s)},\theta
}^{2})\theta \wedge d\theta \\
&=&\int_{S^{3}}R_{0}^{2}\frac{\hat{\theta}\wedge d\hat{\theta}}{R_{0}^{2}}%
=\int_{S^{3}}8\cdot dv_{S^{3}}^{Eucl}=16\pi ^{2}.
\end{eqnarray*}

\noindent So, Rossi spheres are all critical points of $E(J,\theta )$ (as
shown after the proof of Proposition \ref{P-1stV}) with the same energy.

\bigskip

{\center{\Large {Appendix}}}

\medskip

We give some examples for second variations in $\theta $ and $J$ of $E$ at
the standard pseudohermitian $3$-sphere ($S^{3},$ $J_{(0)},$ $\hat{\theta}).$
Substituting $\hat{R}$ $=$ $1$ and $\hat{A}_{11}$ $=$ $0$ into (\ref{S8})
gives%
\begin{equation}
\delta _{\theta }^{2}E(J_{(0)},\hat{\theta})=\int_{S^{3}}(-16|\nabla
_{b}h|^{2}+32(\triangle _{b}h)^{2}-8|h_{,11}|^{2})\hat{\theta}\wedge d\hat{%
\theta}.  \label{A2}
\end{equation}

\noindent Using integration by parts and the commutation relation $h_{,\bar{1%
}1\bar{1}}$ $-$ $h_{,\bar{1}\bar{1}1}$ $=$ $ih_{,\bar{1}0}$ $-$ $h_{,\bar{1}%
} $ (\cite[formula (9)]{CMY}), we compute (volume form $\hat{\theta}\wedge d%
\hat{\theta}$ omitted)%
\begin{eqnarray}
\int_{S^{3}}h_{,11}h_{,\bar{1}\bar{1}} &=&-\int_{S^{3}}h_{,1}(h_{,\bar{1}1%
\bar{1}}-ih_{,\bar{1}0}+h_{,\bar{1}})  \label{A3} \\
&=&-\int_{S^{3}}h_{,1}(h_{,\bar{1}1\bar{1}}-ih_{,0\bar{1}}+h_{,\bar{1}})%
\text{ (since }\hat{A}_{11}=0)  \notag \\
&=&\int_{S^{3}}h_{,1\bar{1}}h_{,\bar{1}1}-i\int_{S^{3}}h_{,1\bar{1}%
}h_{,0}-\int_{S^{3}}|h_{,1}|^{2}.  \notag
\end{eqnarray}

\noindent Since $h_{,1\bar{1}}-h_{,\bar{1}1}=ih_{,0},$ we have 
\begin{eqnarray}
h_{,1\bar{1}} &=&\frac{1}{2}(\triangle _{b}h+ih_{,0})\text{ and}  \label{A4}
\\
h_{,\bar{1}1} &=&\frac{1}{2}(\triangle _{b}h-ih_{,0})\text{ (since }h\text{
is real)}  \notag
\end{eqnarray}

\noindent Substituting (\ref{A4}) into (\ref{A3}), we can reduce (\ref{A3})
to 
\begin{equation}
\int_{S^{3}}h_{,11}h_{,\bar{1}\bar{1}}=\frac{1}{4}\int_{S^{3}}(\triangle
_{b}h)^{2}+\frac{3}{4}\int_{S^{3}}(h_{,0})^{2}-\frac{i}{2}%
\int_{S^{3}}(\triangle _{b}h)h_{,0}-\frac{1}{2}\int_{S^{3}}|\nabla
_{b}h|^{2}.  \label{A4-1}
\end{equation}

\noindent Let $H_{p,q,1}$ denote the restriction to $S^{3}$ of the space of
the homogeneous complex harmonic polynomials of degree $p+q,$ where $p$ is
the holomorphic homogeneity and $q$ the antiholomorphic one. Then for $f$ $%
\in $ $H_{p,q,1}$ one has%
\begin{eqnarray}
-\triangle _{b}f &=&\frac{1}{2}(pq+\frac{1}{2}(p+q))f,  \label{A5} \\
Tf &=&i\frac{(p-q)}{2}f  \notag
\end{eqnarray}%
\noindent (see \cite[Proposition 2.2 on p.10]{ACMY}; note that $\hat{\theta}$
is twice the contact form $\theta _{0}$ in \cite[p.8]{ACMY}, so the
sublaplacian there is twice the sublaplacian here while $T$ there is exactly
the Reeb vector field here, i.e. $T=\hat{T}$). By (\ref{A4-1}) one reduces (%
\ref{A2}) to%
\begin{equation}
\delta _{\theta }^{2}E(J_{(0)},\hat{\theta})=30\int_{S^{3}}(\triangle
_{b}h)^{2}-6\int_{S^{3}}(h_{,0})^{2}+4i\int_{S^{3}}(\triangle
_{b}h)h_{,0}-12\int_{S^{3}}|\nabla _{b}h|^{2}.  \label{A6}
\end{equation}

\noindent Taking $h=f+\bar{f}$ in (\ref{A5}) and writing $\lambda =\frac{1}{2%
}(pq+\frac{1}{2}(p+q)),$ $\mu =\frac{p-q}{2},$ we have -$\triangle
_{b}h=\lambda (f+\bar{f})$ and $h_{,0}$ $=$ $Th$ $=$ $i\mu (f-\bar{f}).$
Substituting these formulas into (\ref{A6}) and noting that $%
\int_{S^{3}}|\nabla _{b}h|^{2}$ $=$ $-\int_{S^{3}}(\triangle _{b}h)h,$ we
reduce the RHS of (\ref{A6}) to%
\begin{equation}
60\int_{S^{3}}\triangle _{b}f\triangle _{b}\bar{f}-12\mu ^{2}\int_{S^{3}}f%
\bar{f}+4i\int_{S^{3}}\triangle _{b}(f+\bar{f})i\mu (f-\bar{f})-24\lambda
\int_{S^{3}}f\bar{f}.  \label{A7}
\end{equation}

\noindent Here we have used $\int_{S^{3}}f^{2}=0,$ $\int_{S^{3}}\bar{f}%
^{2}=0.$ Using (\ref{A5}) we can further reduce (\ref{A7}) and conclude from
(\ref{A6}) that 
\begin{eqnarray}
&&\delta _{\theta }^{2}E(J_{(0)},\hat{\theta})  \label{A8} \\
&=&\int_{S^{3}}\left\{ 
\begin{array}{c}
\frac{60}{4}(pq+\frac{1}{2}(p+q))^{2}-\frac{12}{4}(p-q)^{2} \\ 
-\frac{24}{2}(pq+\frac{1}{2}(p+q))%
\end{array}%
\right\} |f|^{2}\hat{\theta}\wedge d\hat{\theta}  \notag
\end{eqnarray}

\noindent for $\delta \theta =2h\theta =2(f+\bar{f})\hat{\theta}$, $f$ $\in $
$H_{p,q,1}.$ We now turn to compute $\delta _{J}^{2}E(J_{(0)},\hat{\theta}) $
for $\delta J$ $=$ $2E,$ $E=E_{11}\theta ^{1}\otimes Z_{\bar{1}}+$conjugate
with $E_{11}$ $\in $ $H_{p,q,1}.$ Starting from (\ref{S5}) with $\hat{R}$ $=$
$1$ and $\hat{A}_{11}$ $=$ $0,$ we have (volume form $\hat{\theta}\wedge d%
\hat{\theta}$ omitted) 
\begin{eqnarray}
&&\delta _{J}^{2}E(J_{(0)},\hat{\theta})  \label{A9} \\
&=&\int_{S^{3}}(-2|E_{11,0}|^{2}-4|E_{\bar{1}\bar{1}%
,1}|^{2}+4|E_{11,1}|^{2}+8|E_{11}|^{2}).  \notag
\end{eqnarray}

\noindent Via an integration by parts and the commutation relation $E_{11,%
\bar{1}1}-E_{11,1\bar{1}}$ $=$ $-iE_{11,0}-2E_{11}$ (noting that $\hat{R}$ $%
= $ $1),$ we reduce the RHS of (\ref{A9}) to%
\begin{equation}
\int_{S^{3}}(2E_{11,00}-4iE_{11,0})E_{\bar{1}\bar{1}\text{ }}(\text{note
that }8|E_{11}|^{2}\text{ is cancelled).}  \label{A10}
\end{equation}

\noindent We compute%
\begin{eqnarray}
E_{11,0} &=&TE_{11}-2\omega _{1}^{1}(T)E_{11}  \label{A11} \\
&=&TE_{11}-2(-i)E_{11}\text{ (}\omega _{1}^{1}=-i\hat{\theta}\text{ by (\ref%
{w11}))}  \notag \\
&=&i\left( \frac{p-q}{2}\right) E_{11}+2iE_{11}=i(\frac{m}{2}+2)E_{11}. 
\notag
\end{eqnarray}

\noindent Here we have written $m=p-q.$ Applying (\ref{A11}) to (\ref{A10})
we finally obtain%
\begin{equation}
\delta _{J}^{2}E(J_{(0)},\hat{\theta})=-\frac{1}{2}m(m+4)%
\int_{S^{3}}|E_{11}|^{2}\hat{\theta}\wedge d\hat{\theta}  \label{A12}
\end{equation}

\noindent (recall that $E_{11}$ $\in $ $H_{p,q,1},$ $m=p-q).$


\begin{thebibliography}{99}
\bibitem{ACMY} Afeltra, C., Cheng, J.-H., Malchiodi, A. and Yang, P., 
\textit{On the variation of the Einstein-Hilbert action in pseudohermitian
geometry (with an appendix by Xiaodong Wang)}, arXiv: 2306.07523v1.

\bibitem{BJ} Bor, G. and Jacobowitz, H., \textit{Left-invariant CR
structures on 3-dimensional Lie groups}, Complex Analysis and its Synergies
7(3), Dec. 2021.

\bibitem{Cartan} Cartan, \'{E}., \textit{Sur la g\'{e}om\'{e}trie
pseudo-conforme des hypersurfaces de deux variables complexes I, }Ann. Math.
Pura Appl. 11.4 (1932) 17-90.

\bibitem{CY} Case, J. and Yang, P., \textit{A Paneitze-type operator for CR
pluriharmonic functions}, Bull. Inst. Math. Academia Sinica (N. S.) 8 (2013)
285-322.

\bibitem{CS} Chen, S.-C. and Shaw, M.-C., \textit{Partial differential
equations in several complex variables}, AMS/IP Studies in Advanced
Mathematics, vol. 19, Amer. Math. Soc., Providence, RI; International Press,
Boston, MA, 2001.

\bibitem{CL} Cheng, J.-H. and Lee, J., \textit{The Burns-Epstein invariant
and deformation of CR structures}, Duke Math. J. 60 (1990) 221-254.

\bibitem{CMY} Cheng, J.-H., Malchiodi, A. and Yang, P., \textit{A positive
mass theorem in three dimensional Cauchy-Riemann geometry}, Advances in
Mathematics 308, 276-347, 2017.

\bibitem{CMY1} Cheng, J.-H., Malchiodi, A. and Yang, P., \textit{On the
Sobolev quotient of three-dimensional CR manifolds}, Revista Matematica
Iberoamericana, to appear. (arXiv: 1904.04665)\textit{.}

\bibitem{Fal} Falbel, E. and Gorodski, C., \textit{On contact sub-riemannian
symmetric spaces,} Annales scientifiques de l' E. N. S. 4$^{e}$ s\'{e}rie,
tome 28, n$^{o}$ 5 (1995) 571-589.

\bibitem{GV} Gursky, M. and Viaclovsky, J., \textit{A new variational
characterization of three-dimensional space forms, }Invent. Math. 145 (2001)
251--278.

\bibitem{JL} Jerison, D. and Lee, J., \textit{Intrinsic CR normal
coordinates and the CR Yamabe problem}, J. Diff. Geom. 29 (1989) 303-343.

\bibitem{KN} Kobayashi, S. and Nomizu, K., \textit{Foundations of
differential geometry} Vol. I, second printing, 1964.

\bibitem{Lee} Lee, J. M., \textit{The Fefferman metric and pseudohermitian
invariants}, Trans A.M.S. 296 (1986) 411-429.

\bibitem{Lee2} Lee, J. M., \textit{Pseudo-Einstein structures on CR manifolds%
}, Amer. J. Math. 110 (1988) 157-178.

\bibitem{Str} Strichartz, R. S., \textit{Sub-Riemannian geometry}, J. Diff.
Geom. 24 (1986) 221-263.

\bibitem{Tan} Tanaka, N., \textit{A differential geometric study on strongly
pseudo-convex manifolds}, Kinokuniya Book Store Co., Ltd, Kyoto, 1975.
\end{thebibliography}
\end{document}